\documentclass[bj]{imsart}
\usepackage{bbm}
\usepackage{hyperref}
\usepackage{amsfonts}
\usepackage{algorithm}
\usepackage{algorithmic}
\usepackage{latexsym, amssymb, amsmath, amscd, amsthm, amsxtra}
\usepackage{mathtools}
\usepackage{enumerate}
\usepackage[all]{xy}
\usepackage{mathrsfs}
\usepackage{fancyhdr}
\usepackage{listings}
\usepackage{hyperref}
\usepackage{enumitem}
\theoremstyle{plain}
\newtheorem{thm}{Theorem}
\newtheorem{defn}{Definition}
\newtheorem{prop}[thm]{Proposition}

\newtheorem{rem}[thm]{Remark}

\newcommand{\Eb}{\mathbb{E}}
\newcommand{\Pb}{\mathbb{P}}
\begin{document}

\begin{frontmatter}
\title{Percolation threshold for metric graph loop soup}
\runtitle{Percolation threshold for metric graph loop soup}
\begin{aug}
\author[A]{\inits{Y.}\fnms{Yinshan}~\snm{Chang}\ead[label=e1]{ychang@scu.edu.cn}}
\author[B]{\inits{H.}\fnms{Hang}~\snm{Du}\ead[label=e2]{duhang@pku.edu.cn}}
\author[C]{\inits{X.}\fnms{Xinyi}~\snm{Li}\ead[label=e3]{xinyili@bicmr.pku.edu.cn}}

\address[A]{College of Mathematics, Sichuan University\printead[presep={,\ }]{e1}}

\address[B]{School of Mathematical Sciences, Peking University\printead[presep={,\ }]{e2}}

\address[C]{Beijing International Center for Mathematical Research, Peking University\printead[presep={,\ }]{e3}}
\end{aug}

\begin{abstract}
In this short note, we show that the critical threshold for the percolation of metric graph loop soup on a large class of transient metric graphs (including quasi-transitive graphs such as $\mathbb{Z}^d$, $d\geq 3$) is $1/2$.
\end{abstract}

\begin{keyword}[class=MSC2020]
	\kwd{60K35}
	\kwd{82B43}
\end{keyword}

\begin{keyword}
\kwd{Percolation threshold}
\kwd{loop soup}
\kwd{metric graph}
\end{keyword}\end{frontmatter}

\section{Introduction and the main result}
The model of Brownian loop soup was first introduced by Symanzik (see e.g.\ \cite{Symanzik}) and revived in \cite{LawWer04} as a Poissonian collection of loops whose law is based on that of the Brownian motion. Its random walk analogue, the random walk loop soup, was introduced in \cite{LawTru07}. Loop soups are intimately related to various objects of interest in probability and statistical physics, in particular via the isomorphism theorem in \cite{LeJ11} linking the loop soup of intensity\footnote{Note that in the early literature there is an inconsistency of a multiplicative factor of $2$ in the intensity parameter from the definition of loop soups; see e.g.\ \cite{Lupu16-2} for a detailed discussion on this issue.} $1/2$ to the Gaussian free field. 

The percolation of loop soups was already considered in \cite{LawWer04} and then  in \cite{SheWer12} under the setting of the two-dimensional Brownian loop soup. The latter paper, among other results, identified the value of the critical intensity as $1/2$. Subsequently, the works \cite{LJL13}, \cite{ChS16}, and \cite{Chang17} considered  percolation for random walk loop soups on $\mathbb{Z}^d$ for $d\geq 3$ and established various results regarding the phase transition in percolative properties. 

In \cite{Lupu16}, Lupu  considered the loop soup on so-called ``metric graphs'' (also referred to as the ``cable system'' or ``cable graphs''), a notion that corresponds to the extension of discrete graphs to a continuous metric space in which each edge of the graph have a ``length'' and Markov chains are embedded in Brownian motions moving continuously along edges. This particular model (referred to in this note as the {\it metric graph loop soup}) interpolates between the discrete and the continuum, on which the power of the isomorphism theorem is maximized, yielding exact formulas (in particular the two-point function from \cite{Lupu16}; see Proposition~ \ref{prop-two-point-function} below for more details) allowing Lupu to conclude that the critical threshold for the percolation of both the random walk loop soup and metric graph loop soup is greater or equal to $1/2$ on 
\begin{itemize}
\item the integer lattice $\mathbb{Z}^d$, $d\geq3$,
\item $\mathbb{Z}^2$ with constant killing, and 
\item the upper half plane $\mathbb{Z} \times \mathbb{N}$. (Lupu also showed that the threshold is indeed $1/2$ in this case in a separate work \cite{Lupu16-2}).
\end{itemize}
 The work \cite{DPR22} extended Lupu's results to a much wider class of transient metric graphs (including non-amenable graphs such as regular trees). 
 
 It remained an open question whether the threshold is exactly equal to $1/2$ for metric graph loop soup on various types of graphs. (In contrast, the critical threshold for the discrete loop percolation does not equal to $1/2$ in general, see e.g.\ \cite[Theorem~1.3]{ChS16} where it is shown that the threshold on $\mathbb{Z}^d$ tends to infinity as $d\to \infty$.)

In this short note, we give a positive answer to this question on a sufficiently general class of metric graphs by a simple application of the Russo's formula and the two-point function discovered in \cite{Lupu16}. 

We now state our main result. Given a metric graph ${\cal G}$, we denote by $\Pb_\alpha$ for the law of the metric graph loop soup on ${\cal G}$ with intensity $\alpha>0$ and by $x_o\longleftrightarrow \infty$ the event that $x_o\in {\cal G}$ is in an unbounded\footnote{w.r.t.\ the graph distance $\operatorname{d}(\cdot,\cdot)$; see Definition \ref{def:mg}.} cluster formed by the loop soup.
\begin{thm}\label{thm:main}
	For any quasi-transitive transient metric graph ${\cal G}$  and any $x_o\in {\cal G}$, it holds that 
	\begin{equation}
		\alpha>1/2 \, \Longrightarrow \Pb_\alpha[x_o\longleftrightarrow \infty]>0.
	\end{equation}
\end{thm}

	Along with the fact that $\Pb_{1/2}[x_o\longleftrightarrow \infty]=0$ (which follows from \cite{Lupu16} and \cite{DPR22}; see Proposition \ref{prop:noperco} for details), this implies that  $\alpha_c$, the critical threshold on ${\cal G}$, is indeed $1/2$.
\begin{rem}
In fact, our proof works for more general metric graphs, see Remark \ref{rem:gen} for discussions on generalizations of  our result.
\end{rem}

\section{Preliminaries}\label{sec:prelim}
In this section, we briefly introduce metric graphs as well as the associated loop soups and discuss a few classical preliminary facts that will be useful to the proof. For a more detailed introduction of metric graphs and related objects, see e.g.\ Section~2 of \cite{Lupu16}.

\smallskip

We start with metric graphs.
\begin{defn}[Metric graphs]\label{def:mg}
Let ${\cal G}^{\rm skeleton}=(V,E,\lambda)$ be an unoriented finite or countably-infinite weighted connected graph of finite degrees such that 
$$
\mbox{$\lambda_{x,y}=\lambda_{y,x}>0$\, for any $(xy)\in E$\quad and \quad $w(x):=\sum_{y\sim x}\lambda_{x,y}<\infty$\; for any $x\in V$}.
$$ Then the metric graph\footnote{Although in the literature metric graphs are usually denoted with a tilde (e.g., ``$\widetilde{{\cal G}}$'' in contrast to its skeleton $\cal G$), in this note we do not follow this convention as we will be almost exclusively working on metric graphs.} ${\cal G}$ associated with ${\cal G}^{\rm skeleton}$ is the metric space where each edge ${(xy)}\in E$ is regarded as an interval of length $(2\lambda_{x,y})^{-1}$, referred to as the {\it metric graph} ${\cal G}$ with skeleton ${\cal G}^{\rm skeleton}$. Given metric graph $\cal G$, we write $V({\cal G})$ and $E({\cal G})$ for the set of vertices and edges of $G$, and for $x,y\in V({\cal G})$ we write $\operatorname{d}(x,y)$ for the {\rm discrete} graph distance on ${\cal G}^{\rm skeleton}$.
\end{defn}
If in addition there is a finite subset $V_o\subset V=V({\cal G})$ such that for any $x\in V$ there is an automorphism of ${\cal G}^{\text{skeleton}}$ mapping $x$ to some $x_o\in V_o$ (this automorphism should also preserve the weights $(\lambda_{x,y})_{(xy)\in E}$), we say that ${\cal G}$ is {\it quasi-transitive}. Examples of quasi-transitive metric graphs include the metric version of periodic lattices (in particular, $\mathbb{Z}^d$) and regular trees with a fixed set of possible choices of edge weights.

\smallskip

We now turn to the Brownian motion and the metric graph  loop soup.

Given a metric graph ${\cal G}$, there is an associated canonical diffusion process, called the Brownian motion $(B^{\cal G}_t)_{t\geq 0}$ on ${\cal G}$. We write $(l_y)_{y\in {\cal G}}$, $(Q_y)_{y\in {\cal G}}$ and $(E_y)_{y\in {\cal G}}$ for the corresponding local time process, the probability law and expectation respectively. 
If ${\cal G}$ is transient, it is possible to define the Green's function $G(\cdot,\cdot):{\cal G}\times {\cal G} \to \mathbb{R}^+$ associated with  $(B^{\cal G}_t)_{t\geq 0}$  as
$$
G(x,y)=  E_x[l_y],\, \forall x,y\in {\cal G},
$$
i.e., the expected value of the local time at $y$ of a Brownian motion  $(B^{\cal G}_t)_{t\geq 0}$ on ${\cal G}$ starting from $x$.  When $x,y\in V({\cal G})$, $G(x,y)$ coincides with the Green's function associated with the Markov jump process on ${\cal G}^{\rm skeleton}$. For any open subset $K$ of ${\cal G}$, it is also possible to define the killed Green's function $G_K(x,y)$ for $x,y\in K$, as the expected local time at $y$ when the Brownian motion started from $x$ exits $K$ for the first time.

\begin{defn}[The metric graph loop soup]
Given a metric graph ${\cal G}$, we can endow a $\sigma$-finite measure $\mu$ on ${\cal L}_{\cal G}$, the space of (rooted) loops on $\cal G$ (i.e.\ continuous paths $l:[0,T] \to {\cal G}$ such that $l(0)=l(T)$) defined via
$$
\mu=\int_{x\in {\cal G}} \int_{0}^\infty Q_{x,x}^t p_t(x,x) \frac{dt}{t} dx
$$
where we denote by $Q^t_{x,x}$ and $p_t(x,x)$ the bridge probability measure and transition kernel of   $(B^{\cal G}_t)_{t\geq 0}$ from $x$ to $x$ of duration $t$ respectively. The metric graph  loop soup with intensity $\alpha>0$ on ${\cal G}$, denoted by $\operatorname{MGLS}_\alpha$ is a Poisson point process on ${\cal L}_{\cal G}$ with intensity $\alpha\mu$. We denote by $\Pb_\alpha=\Pb_\alpha^{\cal G}$ its probability measure and refer to the support of the point process as the configuration $\omega$, which we regard as a subset of ${\cal L}_{\cal G}$. We denote by $\operatorname{tr}(\omega)$ the trace of $\omega$, i.e., the union of the ranges of all loops in $\omega$. 
\end{defn}

\smallskip

We now turn to loop percolation. For two Borel subsets $A,B$ of ${\cal G}$ and a configuration $\omega\sim \operatorname{MGLS}_\alpha$, we say $A {\longleftrightarrow} B$ if $A,B$ are connected by a path entirely lying in $\operatorname{tr}(\omega)$. With slight abuse of notation we also write events like $x\longleftrightarrow \cdot$ as a shorthand for $\{x\}\longleftrightarrow \cdot$.  Furthermore, for any Borel subset $K\subset {\cal G}$,  the event such that $A\cap K$ is connected to $B\cap K$ via loops of $\omega$ {\it entirely} lying in $K$ is denoted by $A\stackrel{K}{\longleftrightarrow}B$. Given $x_o\in {\cal G}$, write $\{x_o\longleftrightarrow \infty\}$ for the event that $x_o$ lies in a cluster of the $\operatorname{tr}(\omega)$ which is unbounded w.r.t.\ the graph distance $\operatorname{d}(\cdot,\cdot)$.

Before proving Theorem \ref{thm:main}, we first briefly discuss the lower bound on percolation threshold. In \cite{DPR22}, a criterion is introduced to determine the (sub-)criticality of metric-graph loop percolation at intensity $1/2$. We say that a metric graph ${\cal G}$ satisfies the ${\rm (Cap)}$ condition if
\begin{equation}\label{eq:cap}
{\rm cap}(A)=\infty\quad\mbox{ for all (d-)unbounded, closed connected set } A\subset {\cal G}.
\end{equation}
Here ${\rm cap}(\cdot)$ stands for the metric-graph capacity, defined in terms of a variational problem as
$$
{\rm cap}(A) := \left( \inf_{\mu} \int_{A\times A} G(x,y) \mu(dx) \mu(dy) \right)^{-1}\,,
$$
where the infimum runs over all probability measures on $A$. 

For quasi-transitive metric graphs (in fact, for any metric graph satisfying the $\rm (Cap)$ condition), an important consequence of \cite{DPR22} and \cite{Lupu16} is that no percolation occurs for $\operatorname{MGLS}_{1/2}$. We summarize this in the following proposition.
\begin{prop}[No percolation at intensity $1/2$]\label{prop:noperco}
For any quasi-transitive transient metric graph ${\cal G}$  and any $x_o\in {\cal G}$, 
$$\Pb_{1/2}[x_o\longleftrightarrow \infty]=0.$$
\end{prop}
\begin{proof}
By quasi-transitivity, $\inf_{x\in {\cal G}} G(x,x)>0.$
Hence by Lemma 3.4 (2) in \cite{DPR22}, one sees that $\cal G$ does satisfy the ${\rm (Cap)}$ condition. Theorem 1.1, ibid.\  then implies that the sign cluster (of the metric graph Gaussian free field on $\cal G$) that contains $x_o$ is a.s.\ bounded. By Lupu's isomorphism theorem (Theorem 1 in \cite{Lupu16}), this implies that the cluster containing $x_o$ in $\operatorname{MGLS}_{1/2}$ is also a.s.\ bounded. The claim hence follows.
\end{proof}

In fact, much more is known on percolative properties of $\operatorname{MGLS}_{1/2}$. In particular, the following formula gives the precise two-point connectivity in terms of Green's function. It is a paraphrase of Proposition 5.2 of \cite{Lupu16}.
\begin{prop}[Killed two-point function]\label{prop-two-point-function}
	For any open subset $K\subset {\cal G}$ and $x,y\in K$, it holds that
	\begin{equation}\label{eq-two-point-function}
		\Pb_{1/2}[x\stackrel{K}{\longleftrightarrow}y]=\frac{2}{\pi}\arcsin\frac{G_K(x,y)}{\sqrt{G_K(x,x)G_K(y,y)}}\,.
	\end{equation}

\end{prop}

We now turn to the proof of Theorem \ref{thm:main}.
To facilitate our analysis, we now consider a noised version of loop percolation which is stochastically dominated by loop percolation. 

Let $E^2({\cal G})$ stand for the set of 2-bonds, i.e., pairs of {\it neighboring} edges in ${\cal G}$. We now consider a ``Bernoulli 2-bond percolation model'' on ${\cal G}$ in which we independently open any 2-bond in  $E^2({\cal G})$ with probability $p>0$.  We denote by $P_p$ the corresponding probability measure and refer to the collection of open 2-bonds also as configurations. For a 2-bond configuration $\omega^{\rm b}$, we write $\operatorname{tr}(\omega^{\rm b})$ for the union of the  (closed) intervals corresponding to the 2-bonds  from $\omega^{\rm b}$, which forms a closed subset of $\cal G$.

For any $\alpha>1/2$, we note that a configuration $\omega\sim \operatorname{MGLS}_\alpha$ has the same distribution as $\omega_1 \cup \omega_2$, where $\omega_1,\,\omega_2$ are independent configurations sampled from $\Pb_{1/2}$ and $\Pb_{\alpha-1/2}$, respectively.
By quasi-transitivity of ${\cal G}$, it is clear that $\operatorname{tr}(\omega_2)$ stochastically dominates $ \operatorname{tr}(\omega^{\rm b}_p)$ where $\omega^{\rm b}_p \sim P_p$ for some $p=p(\alpha-1/2,{\cal G})>0$.

We now consider the superposition of the metric graph loop percolation and Bernoulli 2-bond percolation. We denote by $\overline{\Pb}_{\varepsilon}$ the law of the union of two independent configurations $\omega_1\sim \Pb_{1/2}$ and $\omega^{\rm b}_\varepsilon\sim P_\varepsilon$ and by $\overline{\Eb}_{\varepsilon}$ the corresponding expectation. For this model, we regard the  configuration $\omega=\omega_1\cup \omega^{\rm b}_\varepsilon$ as a subset of ${\cal L}^{\cal G} \cup E^2({\cal G})$ and let $\operatorname{tr}(\omega)=\operatorname{tr}( \omega_1)\cup\operatorname{tr}(  \omega^{\rm b}_\varepsilon)$ stand for the trace of $\omega$. Note that here we choose this particular model (instead of the usual Bernoulli bond percolation) for a technical reason that arises when one considers metric graphs; see Remark \ref{rem:2bond} for relevant discussions.


We now verify Russo's formula for this noised loop percolation model. For any event $A$, we say that it is {\it 2-bond finite-range}, if there exists a finite subset $S$ of $E^2({\cal G})$ such that the occurrence of  $A$ does not depend on 2-bonds in $E^2({\cal G})\setminus S$. For an increasing event $A$ and a configuration $\omega$, we say a 2-bond $ef\in E^2({\cal G})$  is {\it upper semi-pivotal} (but abbreviated just as ``pivotal'' below for brevity) for $A$ in $\omega$, if $\omega \cup \{ef\}\in A$ but $\omega \notin A$.  Note that given $A$, the event ``$ef$ is pivotal for $A$ in $\omega$'' is measurable with respect to $\omega$. 

\begin{prop}[Russo's formula]\label{prop-Russo}
	For any $\varepsilon \in [0,1/2)$, and any increasing and 2-bond finite-range event $A$,  it holds that as $\delta\downarrow 0$, 
	\begin{equation}
	\overline{\Pb}_{\varepsilon+\delta}[A]-\overline{\Pb}_\varepsilon[A]=\frac{\delta}{1-\varepsilon}\overline{\Eb}_\varepsilon\Big[\#\{ef\in E^2({\cal G}):\;ef\mbox{ is pivotal for }A\}\Big]+O\big(\delta^2\big).
	\end{equation}
\end{prop}
Here we require $\varepsilon \in [0,1/2)$ to avoid dependence of constants from $O(\cdot)$ on $\varepsilon$.
\begin{proof}
Let $\omega \sim \overline{\Pb}_{\varepsilon}$ and $\omega' \sim P_{\delta/(1-\varepsilon)}$ be independent and write $\omega^+=\omega \cup \omega'$. Then, $\omega^+ \sim \overline{\Pb}_{\varepsilon+\delta}$. Hence,
\begin{equation*}
\begin{split}
\overline{\Pb}_{\varepsilon+\delta}[A]-\overline{\Pb}_\varepsilon[A]\;&= \overline{\Pb}_{\varepsilon} \otimes P_{\delta/(1-\varepsilon)}\Big[ \omega^+\in A\mbox{ but } \omega \notin A\Big ]\\
&= \overline{\Eb}_{\varepsilon}\Big[ P_{\delta/(1-\varepsilon)}[\omega\cup\omega'\in A | \omega] 1_{\omega\notin A}\Big]\\
&=\overline{\Eb}_{\varepsilon} \Big[ \frac{\delta}{1-\varepsilon}\sum_{ef\in E^2({\cal G})} 1_{\omega \cup\{ef\} \in A,\;\omega\notin A}+O\big(\delta^2\big)\Big]\\
&=\frac{\delta}{1-\varepsilon}\overline{\Eb}_\varepsilon\Big[\#\{ef\in E^2({\cal G}):\;ef\mbox{ is pivotal for }A\}\Big]+O\big(\delta^2\big)
\end{split}
\end{equation*} 
where the constant in $O(\delta^2)$ does not depend on $\varepsilon$ (it only depends on $\cal G$ and the event $A$) and in the second-to-last equality we use inclusion-exclusion principle and the finite-range dependence of $A$ as in the proof of the classical Russo's formula for Bernoulli percolation.
\end{proof}
We briefly remark that it is also possible to define pivotal 2-bonds in the following manner. Consider $\omega= \omega_1\cup \omega^{\rm b}$ where $\omega_1\subset {\cal L}_{\cal G}$ is a metric graph loop configuration and $\omega^{\rm b}\subset E^2({\cal G})$ is a 2-bond percolation configuration. We say that a 2-bond $ef$ is pivotal for event $A$ in $\omega$ if $\omega\cup \{ef\}\in A$ but $\omega_1\cup (\omega^{\rm b}\setminus\{ef\}) \notin A$ and in this case the RHS of Russo's formula becomes $$\delta\cdot  \overline{\Eb}_{\varepsilon}[\#\{ef\in  E^2({\cal G}):\,ef\mbox{ is pivotal}\}]+O(\delta^2),$$ which is also sufficient for our analysis.

Finally, we turn to the FKG inequality for the noised loop percolation measure $\overline{\Pb}_\varepsilon$, which is a direct consequence of the FKG inequality for general Poisson processes (see Lemma 2.1 of \cite{Jans84}) and the observation that when looking at percolative properties, one can effectively regard $\omega\sim \overline{\Pb}_\varepsilon$ as a Poisson process in ${\cal L}_{\cal G} \cup E^2({\cal G})$.
\begin{prop}[FKG inequality]\label{prop:FKG}
For $\varepsilon\in[0,1)$ and any increasing events $A$ and $B$,
$$
\overline{\Pb}_\varepsilon[A\cap B]\geq \overline{\Pb}_\varepsilon[A]\overline{\Pb}_\varepsilon[B].
$$
\end{prop}

\section{Proof of the main result}
Without loss of generality we assume $x_o\in V({\cal G})$. Recall that  $\operatorname{d}(\cdot,\cdot)$ stands for the {\it discrete} graph metric. We write $$ B_n=\{x\in V({\cal G}):\operatorname{d}(x_o,x)\leq n\}\;\mbox{ and }\;\partial^{
\rm s} B_n=\{x\in V({\cal G}):\operatorname{d}(x_o,x)=n\}\,.$$
With slight abuse of notation, for $\omega \sim \overline{\Pb}_\varepsilon$ we still denote by $\{ x \longleftrightarrow A\}$ the connectivity event that a point $x$ and a Borel set $A$ are connected through $\operatorname{tr}(\omega)$. We then write 
$$
f_{n}(\varepsilon):=\overline{\Pb}_\varepsilon\big[x_o\longleftrightarrow \partial^{\rm s} B_n\big].
$$
The crux of the proof is the following differential inequality, which is remotely inspired by \cite{DCT16}.
\begin{prop}\label{prop:ODE}
	There exist constants $c,C>0$ (depending on ${\cal G}$ and $x_o$ only), such that for any $n
	\geq 3$, it holds
	\begin{equation}
		f_{n,+}^{\prime}(\varepsilon)\ge c(1-Cf_n(\varepsilon))\,,\quad \forall\  0\le \varepsilon<1/2\,,
	\end{equation}
 where $f_{n,+}'$ stands for the right derivative of $f_n$.
\end{prop}
\begin{proof}
Note that $\{x_o\longleftrightarrow \partial^{\rm s} B_n\}$ is increasing and finite-range. From Proposition~\ref{prop-Russo}, it suffices to show that there exists $c>0$ such that for any $\varepsilon \in [0,1/2)$,  
\[
\overline{\Eb}_\varepsilon\big[ \#\{ef\text{ is pivotal for }x_o\longleftrightarrow \partial^{\rm s} B_n\}\big]\ge c\big(1-Cf_n(\varepsilon)\big)\,.
\]
For any configuration $\omega$ sampled from $\overline{\Pb}_\varepsilon$, denote by $\omega_{B_n}$ its restriction in loops and 2-bonds that intersect $B_n$. We consider the union of connected components of $\operatorname{tr}(\omega_{B_n})$ intersecting $\partial^{\rm s} B_n$, and denote its closure by $C_n=C_n(\omega)$. Note that
\begin{equation}\label{eq:nointersection}
\partial C_n \cap V({\cal G})=\emptyset.
\end{equation}
  Define $K_n$ as the component of $B_n\setminus C_n$ containing $x_o$ (if $x_o\in C_n$, then set $K_n=\emptyset$).

\begin{figure}[!ht]
\centering
\includegraphics[scale=.85]{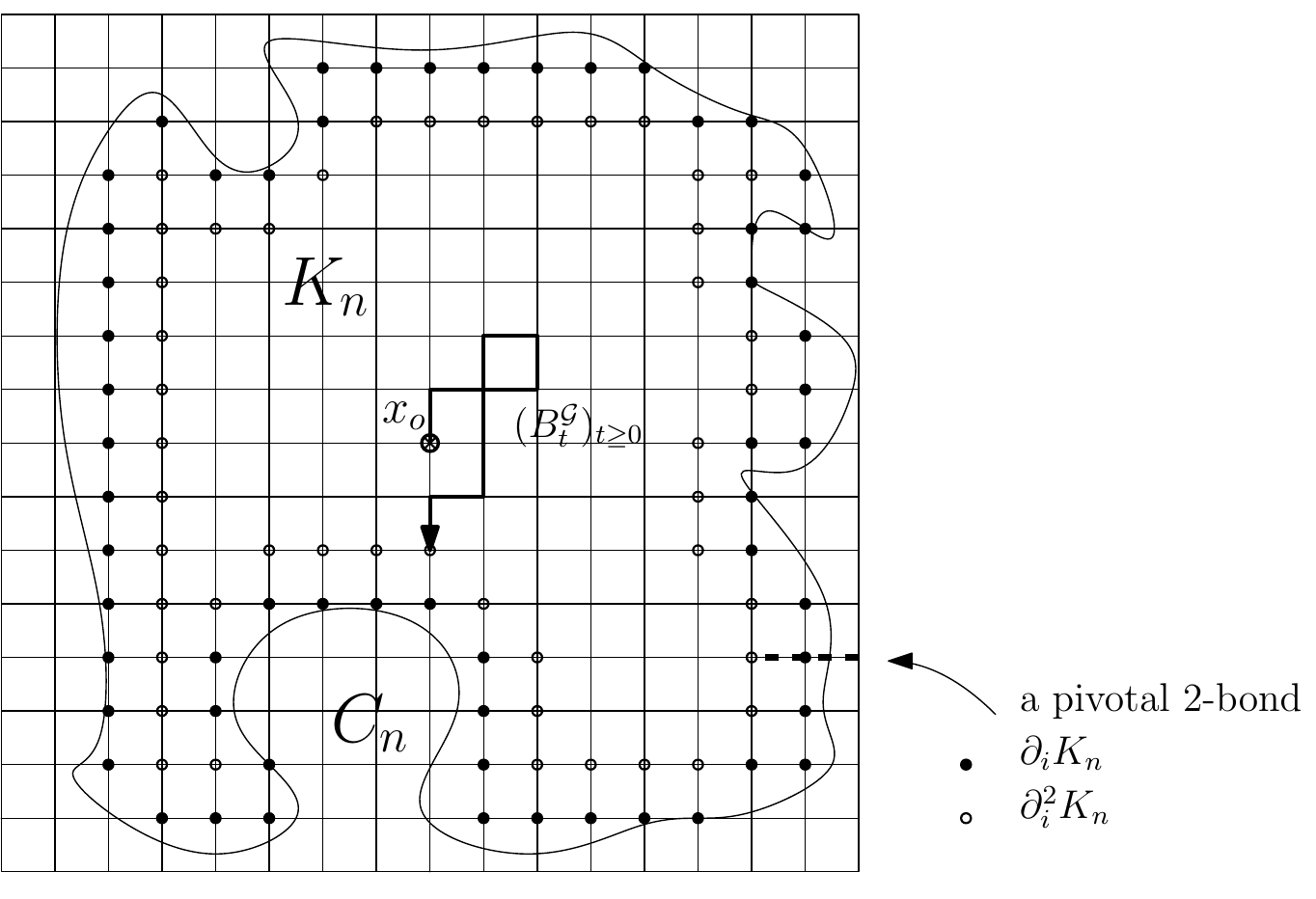}
\caption{\small A possible realization of $K_n$, $C_n$, $\partial_i K_n$, $\partial^2_i K_n$ and an exemplary pivotal 2-bond. Note that in this figure $B_n$ and $C_n$ are only partially depicted.}
\label{fig1}
\end{figure}

We now explore $C_n$ from $\partial^{\rm s} B_n$ inwards.   We claim the following strong Markov property: conditioned on the realization of $C_n$, the configuration in $K_n$ is the independent union of the loop soup of intensity $1/2$ in $K_n$ and an independent Bernoulli 2-bond percolation with parameter $\varepsilon$ on $K_n^{\rm s}:=V({\cal G})\cap K_n$, the skeleton of $K_n$. Moreover, all of them are independent from the configurations in $C_n$. To see this, one may view the noised loop percolation configuration as $\mathbb{R}\times \mathbb{N}$-valued random field\footnote{In other words, we store the occupation time field of the original $\operatorname{MGLS}_{1/2}$ and the total count of 2-bonds on each bond  in the first and second coordinates respectively so on the boundary of $C_n$ both coordinates are zero simultaneously.} and show that it satisfies a simple Markov property and establish strong Markov property through Theorem 4, Chapter 2, Section 2.4 of \cite{rozanov1982markov}. 

%

Write $C_n^s:=V({\cal G})\cap C_n$. Let 
$$
E(K_n,C_n)=\big\{ef=(x,y)(y,z)\in E^2({\cal G})\,:\,x,y\in K_n^{\rm s},\;z\in C_n^s\big\}
$$
be the boundary 2-bonds of $K_n$ and let
$$
\partial_i K_n:=\big\{x\in K_n^s: {\rm d}(x,C^{\rm s}_n)=1\big\}\; \mbox{ and }\;\partial^2_i K_n=\big\{x\in K_n^s: {\rm d}(x,C^{\rm s}_n)=2\big\}
$$ 
stand for the inner vertex boundary and 2-inner vertex boundary of $K_n$ respectively. 
Write 
$$\mathcal A_n:=\{
\operatorname{d}(x_o,C_n^{\rm s})>2\}.
$$
Note that $\mathcal A_n$ is measurable with respect to the realization of $C_n$. See Figure \ref{fig1} for illustration.

With \eqref{eq:nointersection} in mind, we now make the following observation: conditioned on any realization of $C_n$ such that $\mathcal A_n$ holds, for any $ef=(x,y)(y,z)\in E(K_n,C_n)$, $ef$ is pivotal for $x_o\longleftrightarrow \partial^{\rm s} B_n$ if  $x_o\stackrel{K_n}{\longleftrightarrow} x$. As a result we have
	\begin{align}
		&\ \overline{\Eb}_\varepsilon\big[\#\{ef\in E(K_n,C_n);\;ef\text{ is pivotal\ for }x_o\longleftrightarrow \partial^{\rm s} B_n\}\big]\nonumber\\
		\ge&\ \overline{\Eb}_\varepsilon \bigg[\sum_{ef\in E(K_n,C_n)}\overline{\Pb}_\varepsilon\Big[ef \text{ is pivotal\ for }x_o\longleftrightarrow \partial^{\rm s} B_n\Big| C_n\Big]\bigg]\nonumber\\
			\ge&\ \overline{\Pb}_\varepsilon[\mathcal A_n]\cdot \overline{\Eb}_\varepsilon\bigg[ \sum_{ef=(x,y)(y,z)\in E(K_n,C_n),\,x\in \partial^2_i K_n}\overline{\Pb}_\varepsilon\Big[x_o\stackrel{K_n}{\longleftrightarrow} x \Big | C_n\Big]\bigg| {\cal A}_n\bigg]\nonumber\\
		\ge&\ \overline{\Pb}_\varepsilon[\mathcal A_n]\cdot \overline{\Eb}_\varepsilon\bigg[\sum_{x\in \partial^2_i K_n}\Pb_{1/2}\Big[x_o\stackrel{K_n}{\longleftrightarrow} x\Big | C_n\Big]\bigg| {\cal A}_n\bigg]\,.\label{eq-two-terms}
	\end{align}

To estimate the first term in \eqref{eq-two-terms}, for $x\in V({\cal G})$, we define $\mathcal O(x)$ for the event that all edges incident to $x$ and graph neighbors of $x$ are covered by a single loop in $\operatorname{MGLS}_{1/2}$. Also we note that $\mathcal A_n^c$ is equivalent to the event that one of the neighbors or 2-neighbors of $x_o$ in $V({\cal G})$ is connected to $\partial^{\rm s} B_n$. Then it is clear that both $\mathcal A_n^c$ and $\mathcal O(x_o)$ are increasing, and $\mathcal A_n^c\cap \mathcal O(x_o)\subset \{x_o\longleftrightarrow \partial^{\rm s} B_n\}$. By FKG inequality (see Proposition \ref{prop:FKG}), we get
\[
\overline{\Pb}_\varepsilon[\mathcal A_n]=1-\overline{\Pb}_\varepsilon[\mathcal A_n^c]\ge 1-\frac{\overline{\Pb}_\varepsilon[\mathcal A_n^c\cap \mathcal O(x_o)]}{\Pb_{1/2}[\mathcal O(x_o)]}\ge 1-C\overline{\Pb}_\varepsilon[x_o\longleftrightarrow \partial^{\rm s} B_n]=1-Cf_n(\varepsilon),
\] 
where $$C=C({\cal G})=\sup_{x\in V({\cal G})}\left(\Pb_{1/2}[\mathcal O(x)]\right)^{-1}$$ is a positive constant depending on ${\cal G}$ only thanks to quasi-transitivity. For the second term, 
from \eqref{eq-two-point-function} we see that
\begin{equation}\label{eq-est-1}
\begin{aligned}
\sum_{x\in \partial^2_i K_n}\Pb_{1/2}[x_o\stackrel{K_n}{\longleftrightarrow} x\mid C_n]=&\ \sum_{x\in \partial_i^2K_n}\frac{2}{\pi}\arcsin\frac{G_{K_n}(x_o,x)}{\sqrt{G_{K_n}(x_o,x_o)G_{K_n}(x,x)}}\\
\ge&\ c_1\sum_{x\in \partial^2_i K_n} \frac{G_{K_n}(x_o,x)}{\sqrt{G_{K_n}(x,x)}}\,,
\end{aligned}
\end{equation}
where $c_1={2}\big(\pi\sqrt{G(x_o,x_o)}\big)^{-1}$ depends only on ${\cal G}$ and $x_o$, and we have used the fact that $\arcsin x\ge x$ for $x\in (0,1]$ and $G_{K_n}(x_o,x_o)\le G(x_o,x_o)$ for any $K_n\subset{\cal G} $. Furthermore, we have for each $x\in \partial_i^2 K_n$,
\begin{equation}\label{eq-est-2}
	\begin{aligned}
\frac{G_{K_n}(x_o,x)}{\sqrt{G_{K_n}(x,x)}}=&\ \sqrt{G_{K_n}(x,x)}\cdot\frac{G_{K_n}(x_o,x)}{G_{K_n}(x,x)}\\
=&\ \sqrt{G_{K_n}(x,x)}\cdot Q_{x_o}\left[ (B^{\cal G}_t)_{t\geq 0}\text{  hits }x\text{ before exiting }K_n\right]\,.
\end{aligned}
\end{equation}
 In addition, $x\in \partial_i^2 K_n$ implies  $G_{K_n}(x,x)\ge c_2$ for some constant $c_2>0$ depending only on ${\cal G}$.  
 
 Combining \eqref{eq-est-1} with \eqref{eq-est-2}, we see that the second term in \eqref{eq-two-terms} is bounded from below by
\begin{equation}\label{eq:lb}
c_1\sqrt{c_2}\sum_{x\in \partial_i^2K_n}Q_{x_o}
\left[ (B^{\cal G}_t)_{t\geq 0}\text{ hits }x\text{ before leaving }K_n\right]\,.
\end{equation}
Note that under the event $\mathcal A_n$, the sum in the above formula is trivially bounded from below by $1$, since a path of the Brownian motion  $(B^{\cal G}_t)_{t\geq 0}$ on ${\cal G}$ starting from $x_o$ almost surely hits at least one vertex $x\in \partial_i^2 K_n$ before exiting $K_n$. Therefore, the proof is completed by taking $c=c_1\sqrt{c_2}$.
\end{proof}
\begin{rem}\label{rem:2bond}
It is worth noting that there is a difficulty in the argument above if we only consider $\partial_i K_n$, the interior graph boundary of $K_n$ (instead of its 2-inner boundary) in summing the two point functions. Indeed, for some $x\in \partial_i K_n$, it could happen that $x$ is very close to $C_n$ in the metric sense, and in this case one may not bound $G_{K_n}(x,x)$ from below by any universal constant. \end{rem}
\begin{proof}[Proof of Theorem \ref{thm:main}]
The claim follows from Proposition \ref{prop:ODE} and a standard and classical argument. Recall (from the definition of $\overline{\Pb}_\varepsilon$ in  Section \ref{sec:prelim}) that  for any $\alpha>1/2$, there exists some $\varepsilon>0$ such that $\overline{\Pb}_\varepsilon$ is dominated $\Pb_\alpha$. Hence, it suffices to show that for any $\varepsilon\in(0,1/2)$, 
$$f(\varepsilon):=\overline{\Pb}_\varepsilon[x_o\longleftrightarrow \infty]>0.$$
For any $\varepsilon\in(0,1/2)$, if $f(\varepsilon)\geq 1/(2C),$ where $C$ is the constant from Proposition \ref{prop:ODE}, then the claim naively follows; otherwise, one can find $N_0=N_0(\varepsilon) \geq 3$ such that $f_n(\varepsilon)<1/(2C)$ for all $n\geq N_0$. By Proposition \ref{prop:ODE} and the monotonicity of $f_n(\cdot)$, uniformly for all $\delta\in(0,\varepsilon)$ and $n>N_0$, $f'_{n,+}(\delta)\geq c(1-Cf_n(\varepsilon))>c/2$. From this we conclude that $f_n(\varepsilon)\geq c\varepsilon/2$ for any $n\geq N_0 $. Combining two cases, one has $f(\varepsilon)>0$ for any $\varepsilon>0$. 
\end{proof}
\begin{rem}\label{rem:gen} 1) Our argument does not essentially rely on the quasi-transitivity of ${\cal G}$ except at places where we require a uniform constant depending on the graphs ${\cal G}$.  In fact, it applies to all transient weighted metric graphs with uniformly bounded degrees and edge weights uniformly bounded from above (but not necessarily from below -- since one may always break a ``long'' edge into shorter pieces, which produces a new skeleton graph but keep the metric space (of the original metric graph ${\cal G}$) and $\alpha_c$ unchanged; note, however, that in this case we can only establish the existence of an unbounded cluster with respect to the underlying metric of ${\cal G}$, not the graph distance on ${\cal G}^{\rm skeleton}$). \newline
2) Another direction of generalization is to take killing into consideration. This corresponds to massive loop soups or loop soups on a metric graph with boundary). If the killing is mild, i.e.,
\begin{equation}\label{eq:mild}
h_{\rm kill}(x) :=Q_x\big[ (B^{\cal G}_t)_{t\geq 0}\mbox{ is killed }\big]<1,\;\forall x\in G,
\end{equation}
(with slight abuse of notation we still denote by $(B^{\cal G}_t)_{t\geq 0}$ and $Q_x$ the massive Brownian motion and the respective law) then our argument still works with little modification (in this case the sum in \eqref{eq:lb} is bounded below by $1-h_{\rm kill}(x_o)$ instead). This can be compared with a similar conclusion in the case of metric graph GFF; see (1.8) of \cite{DPR22}. When the killing is strong (i.e., $h_{\rm kill}(\cdot)\equiv1$), our argument fails and hence it is very natural to wonder if $\alpha_c=1/2$ still holds for various graphs that fall in this category. Take $\mathbb{Z}^2$ with constant killing for instance. In Theorem 5.2.11 of \cite{camia2016scaling}, it is established that the critical intensity of the massive planar Brownian loop soup is still $1/2$ (in our notation). But one cannot adapt the arguments of \cite{Lupu16-2} in this case because the massive Brownian loop soup only corresponds to discrete (including metric graph) loop soup with diminishing mass in the scaling limit.

However, it is worth pointing out that most likely \eqref{eq:mild} is not a necessary condition for our result to hold. We now provide a example. For the metric graph $\mathbb{Z}^{d}\times\mathbb{N}$, $d\geq 2$ with uniform edge weights and instantaneous killing on the boundary $\mathbb{Z}^{d}\times\{0\}$, it is very plausible that one can show $\alpha_c=1/2$ by imitating the proof for Lemma 1.8 of \cite{Chang17}, which adapts \cite{GrimMas} to establish that the percolation threshold in a slab approximates the full space threshold for random walk loop soups. 

There are also graphs that do not satisfy the (Cap) condition \eqref{eq:cap} for which percolation already occurs at $\alpha=1/2$ (see Section 7 of \cite{Prevost2023} for such an example). It is then very natural to wonder if $\alpha_c<1/2$ in this case.

With all these discussions in mind, let us end this note by asking the following questions: what is the sufficient and necessary condition one should pose on a metric graph for  $\alpha_c$ to be exactly $1/2$? Is the same condition also the sufficient and necessary for the critical threshold of level-set percolation for metric graph GFF to be $0$? 
\end{rem}

\begin{acks}[Acknowledgments]
We thank two anonymous referees for their helpful and detailed comments on an earlier version of the manuscript. XL also wishes to thank Pierre-Francois Rodriguez, Alex Drewitz and Zhenhao Cai for various helpful discussions. 
\end{acks}

\begin{funding}
The authors  acknowledge the support of National Key R\&D Program of China (No.\ 2021YFA1002700 and No.\ 2020YFA0712900). YC acknowledges the support of NSFC (No.\ 11701395). HD is partially supported by the elite undergraduate training program of School of Mathematical Science at Peking University. XL acknowledges the support of NSFC (No.\ 12071012). 
\end{funding}

\bibliography{percolation}

\begin{thebibliography}{17}

\bibitem[\protect\citeauthoryear{Camia}{2016}]{camia2016scaling}
\begin{bincollection}[author]
\bauthor{\bsnm{Camia},~\bfnm{Federico}\binits{F.}}
(\byear{2016}).
\btitle{Scaling limits, Brownian loops and conformal fields}.
In \bbooktitle{Advances in Disordered Systems, Random Processes and Some
  Applications}
(\beditor{\bfnm{Pierluigi}\binits{P.}~\bsnm{Contucci}} \AND
  \beditor{\bfnm{Giardin\`a}\binits{G.}~\bsnm{Cristian}}, eds.)
\bchapter{5},
\bpages{205--269}.
\bpublisher{Cambridge University Press}, \baddress{Cambridge}.
\end{bincollection}
\endbibitem

\bibitem[\protect\citeauthoryear{Chang}{2017}]{Chang17}
\begin{barticle}[author]
\bauthor{\bsnm{Chang},~\bfnm{Yinshan}\binits{Y.}}
(\byear{2017}).
\btitle{Supercritical loop percolation on $\mathbb{Z}^d$ for $d\geq 3$}.
\bjournal{Stochastic Process.\ Appl.}
\bvolume{127}
\bpages{3159--3186}.
\end{barticle}
\endbibitem

\bibitem[\protect\citeauthoryear{Chang and Sapozhnikov}{2016}]{ChS16}
\begin{barticle}[author]
\bauthor{\bsnm{Chang},~\bfnm{Yinshan}\binits{Y.}} \AND
  \bauthor{\bsnm{Sapozhnikov},~\bfnm{Art{\"e}m}\binits{A.}}
(\byear{2016}).
\btitle{Phase transition in loop percolation}.
\bjournal{Probab.\ Theory Relat.\ Fields}
\bvolume{164}
\bpages{979--1025}.
\end{barticle}
\endbibitem

\bibitem[\protect\citeauthoryear{Drewitz, Pr{\'e}vost and
  Rodriguez}{2022}]{DPR22}
\begin{barticle}[author]
\bauthor{\bsnm{Drewitz},~\bfnm{Alexander}\binits{A.}},
  \bauthor{\bsnm{Pr{\'e}vost},~\bfnm{Alexis}\binits{A.}} \AND
  \bauthor{\bsnm{Rodriguez},~\bfnm{Pierre-Fran{\c{c}}ois}\binits{P.-F.}}
(\byear{2022}).
\btitle{Cluster capacity functionals and isomorphism theorems for Gaussian free
  fields}.
\bjournal{Probab.\ Theory Relat.\ Fields}
\bvolume{183}
\bpages{255--313}.
\end{barticle}
\endbibitem

\bibitem[\protect\citeauthoryear{Duminil-Copin and Tassion}{2016}]{DCT16}
\begin{barticle}[author]
\bauthor{\bsnm{Duminil-Copin},~\bfnm{Hugo}\binits{H.}} \AND
  \bauthor{\bsnm{Tassion},~\bfnm{Vincent}\binits{V.}}
(\byear{2016}).
\btitle{A new proof of the sharpness of the phase transition for Bernoulli
  percolation and the Ising model}.
\bjournal{Comm.\ Math.\ Phys.}
\bvolume{343}
\bpages{725--745}.
\end{barticle}
\endbibitem

\bibitem[\protect\citeauthoryear{Grimmett and Marstrand}{1990}]{GrimMas}
\begin{barticle}[author]
\bauthor{\bsnm{Grimmett},~\bfnm{Geoffrey~Richard}\binits{G.~R.}} \AND
  \bauthor{\bsnm{Marstrand},~\bfnm{John~M}\binits{J.~M.}}
(\byear{1990}).
\btitle{The supercritical phase of percolation is well behaved}.
\bjournal{Proc.\ R.\ Soc.\ A: Math.\ Phys.\ Eng.}
\bvolume{430}
\bpages{439--457}.
\end{barticle}
\endbibitem

\bibitem[\protect\citeauthoryear{Janson}{1984}]{Jans84}
\begin{barticle}[author]
\bauthor{\bsnm{Janson},~\bfnm{Svante}\binits{S.}}
(\byear{1984}).
\btitle{Bounds on the distributions of extremal values of a scanning process}.
\bjournal{Stochastic Process.\ Appl.}
\bvolume{18}
\bpages{313--328}.
\end{barticle}
\endbibitem

\bibitem[\protect\citeauthoryear{Lawler and Trujillo~Ferreras}{2007}]{LawTru07}
\begin{barticle}[author]
\bauthor{\bsnm{Lawler},~\bfnm{Gregory}\binits{G.}} \AND
  \bauthor{\bsnm{Trujillo~Ferreras},~\bfnm{Jos{\'e}}\binits{J.}}
(\byear{2007}).
\btitle{Random walk loop soup}.
\bjournal{Trans.\ Amer.\ Math.\ Soc.}
\bvolume{359}
\bpages{767--787}.
\end{barticle}
\endbibitem

\bibitem[\protect\citeauthoryear{Lawler and Werner}{2004}]{LawWer04}
\begin{barticle}[author]
\bauthor{\bsnm{Lawler},~\bfnm{Gregory~F}\binits{G.~F.}} \AND
  \bauthor{\bsnm{Werner},~\bfnm{Wendelin}\binits{W.}}
(\byear{2004}).
\btitle{The Brownian loop soup}.
\bjournal{Probab.\ Theory Relat.\ Fields}
\bvolume{128}
\bpages{565--588}.
\end{barticle}
\endbibitem

\bibitem[\protect\citeauthoryear{Le~Jan}{2011}]{LeJ11}
\begin{bbook}[author]
\bauthor{\bsnm{Le~Jan},~\bfnm{Yves}\binits{Y.}}
(\byear{2011}).
\btitle{Markov Paths, Loops and Fields: {\'E}cole D'{\'E}t{\'e} de
  Probabilit{\'e}s de Saint-Flour XXXVIII--2008}
\bvolume{2026}.
\bpublisher{Springer Science \& Business Media}.
\end{bbook}
\endbibitem

\bibitem[\protect\citeauthoryear{Le~Jan and Lemaire}{2013}]{LJL13}
\begin{barticle}[author]
\bauthor{\bsnm{Le~Jan},~\bfnm{Yves}\binits{Y.}} \AND
  \bauthor{\bsnm{Lemaire},~\bfnm{Sophie}\binits{S.}}
(\byear{2013}).
\btitle{Markovian loop clusters on graphs}.
\bjournal{Illinois J.\ Math.}
\bvolume{57}
\bpages{525--558}.
\end{barticle}
\endbibitem

\bibitem[\protect\citeauthoryear{Lupu}{2016a}]{Lupu16-2}
\begin{barticle}[author]
\bauthor{\bsnm{Lupu},~\bfnm{Titus}\binits{T.}}
(\byear{2016}a).
\btitle{{Loop percolation on discrete half-plane}}.
\bjournal{Electron.\ Commun.\ Probab.}
\bvolume{21}
\bpages{1--9}.
\end{barticle}
\endbibitem

\bibitem[\protect\citeauthoryear{Lupu}{2016b}]{Lupu16}
\begin{barticle}[author]
\bauthor{\bsnm{Lupu},~\bfnm{Titus}\binits{T.}}
(\byear{2016}b).
\btitle{{From loop clusters and random interlacements to the free field}}.
\bjournal{Ann. Probab.}
\bvolume{44}
\bpages{2117--2146}.
\end{barticle}
\endbibitem

\bibitem[\protect\citeauthoryear{Pr{\'e}vost}{2023}]{Prevost2023}
\begin{barticle}[author]
\bauthor{\bsnm{Pr{\'e}vost},~\bfnm{Alexis}\binits{A.}}
(\byear{2023}).
\btitle{Percolation for the Gaussian free field on the cable system:
  counterexamples}.
\bjournal{Electron.\ J.\ Probab.}
\bvolume{28}
\bpages{1--43}.
\end{barticle}
\endbibitem

\bibitem[\protect\citeauthoryear{Rozanov}{1982}]{rozanov1982markov}
\begin{bbook}[author]
\bauthor{\bsnm{Rozanov},~\bfnm{Yu~A}\binits{Y.~A.}}
(\byear{1982}).
\btitle{Markov Random Fields}.
\bpublisher{Springer}, \baddress{New York}.
\end{bbook}
\endbibitem

\bibitem[\protect\citeauthoryear{Sheffield and Werner}{2012}]{SheWer12}
\begin{barticle}[author]
\bauthor{\bsnm{Sheffield},~\bfnm{Scott}\binits{S.}} \AND
  \bauthor{\bsnm{Werner},~\bfnm{Wendelin}\binits{W.}}
(\byear{2012}).
\btitle{Conformal loop ensembles: the Markovian characterization and the
  loop-soup construction}.
\bjournal{Ann.\ of Math.}
\bpages{1827--1917}.
\end{barticle}
\endbibitem

\bibitem[\protect\citeauthoryear{Symanzik}{1967}]{Symanzik}
\begin{bincollection}[author]
\bauthor{\bsnm{Symanzik},~\bfnm{Kurt}\binits{K.}}
(\byear{1967}).
\btitle{Euclidean quantum field theory}.
In \bbooktitle{Local Quantum Theory,}
(\beditor{\bfnm{R.}\binits{R.}~\bsnm{Jost}}, ed.)
\bpages{152--226}.
\bpublisher{Acad.\ Press}, \baddress{New York}.
\end{bincollection}
\endbibitem

\end{thebibliography}
\bibliographystyle{imsart-nameyear}

\end{document}